\documentclass[reqno]{amsart}
\usepackage{stmaryrd}
\usepackage{enumitem}
\usepackage{amssymb,amsmath}
\usepackage{mathabx}
\usepackage{mathbbol}
\def\dispace{\setlength{\itemsep}{2pt}}

\def\pSkip{\vskip 1mm \noindent}

\def\lcmn{{\overline{n}}}
\def\add{\vee}
\def\Add{\bigvee}
\def\mlt{+}
\def\Mlt{\sum}
\def\Id{\operatorname{Id}}
\def\oo{\operatorname{o}}
\def\OO{\operatorname{O}}

\def\p{p}
\def\tN{\mathcal{N}}

\def\UT{\mathcal{U}}

\def\Un{\UT_n(\Trop)}

\def\Mat{\tM}

\def\MnT{\Mat_n(\Trop)}
\def\Mn{\Mat_n}
\def\Un{\tU_n}
\def\UnT{\Un(\Trop)}
\def\Sn{\operatorname{S}_n}

\def\hq{\hat q}
\def\hp{\hat p}
\def\hr{\hat r}

\newcommand{\hlt}[1]{\textbf{#1}}
\newcommand\vx[1]{#1} %{\mathfrak{n}_{#1}}

\def\cyc{\theta}
\def\Cycles{\Theta}
\def\SCycles{\Phi}

\newcommand\lno[2]{\#_{#1}({#2})}

\def\N{\mathbb N}

\def\nxn{n\times n}

\def\w{k}

\def\lcm{\operatorname{lcm}}

\def\per{{\operatorname{per}}}
\def\sep{\operatorname{sep}}

\def\N{\mathbb N}

\def\tS{\mathcal S}

\def\om{\omega}

\def\dl{\delta}

\def\ver{\mathcal V}
\def\arc{\mathcal E}
\def\varX{\mathcal A}
\def\a{a}

\def\wlen{\ell}

\newcommand\thmref[2]{\pSkip\textbf{Theorem #1. }\emph{#2}\vskip 1mm}
\newcommand\propref[2]{\pSkip\textbf{Proposition #1. }\emph{#2}\vskip 1mm}

\newcommand\corref[2]{\pSkip\textbf{Corollary #1. }\emph{#2}\vskip 1mm}

\def\tM{\mathcal M}
\def\tU{\mathcal U}

\newcommand\eval[2]{\left\llbracket  #1 , #2 \, \right  \rrbracket}
\newcommand\substit[2]{\left[ #1 , #2 \, \right]}
\newcommand\sid[2]{\langle #1 , #2 \rangle}

\newcommand\cset[1]{\langle #1 \rangle}

\def\set2{{\cset{2}}}

\def\l2{x^2y^2x}
\def\ll2{yx^2y^2x}

\newcommand{\len}[1]{\operatorname{\ell}(#1)}

\newcommand{\ds}[1]{\ {#1} \ }

\def\one{\mathbb{1}}
\def\zero{\mathbb{0}}

\newtheorem{theorem}{Theorem}[section]
\newtheorem{proposition}[theorem]{Proposition}
\newtheorem{definition}[theorem]{Definition}

\newtheorem{lemma}[theorem]{Lemma}

\newtheorem{notation}[theorem]{Notation}
\newtheorem{corollary}[theorem]{Corollary}

\newtheorem{example}[theorem]{Example}
\newtheorem{remark}[theorem]{Remark}

\newcommand {\junk}[1]{}
\newcommand{\etype}[1]{\renewcommand{\labelenumi}{(#1{enumi})}}
\def\eroman{\etype{\roman} \dispace }

\def\({\left(}
\def\){\right)}

\def\al{\alpha}

\def\gm{\gamma}

\def\e{\varepsilon}
\def\w{\om}
\def\Sem{{\mathbf S}}

\def\minf{-\infty}

\def\Real{\mathbb R}
\def\Trop{\mathbb T}

\def\rk{\operatorname{rk}}

\def\tr{\operatorname{tr}}
\def\mtr{\operatorname{tr}}

\def\trrk{\rk_{\tr}}
\def\frk{\rk_{\operatorname{fc}}}
\def\Auto{{\digr}}
\def\cyclty{\operatorname{cyc}}

\def\minf{-\infty}

\def\tlA{\widetilde A}

\def\subcrit{{\mathcal G}}
\def\Wi{(n-1)^2+1}

\def\Swalk{\mathcal{W}}
\newcommand{\walkslen}[3]{\Swalk^{\; #3}(#1\to #2)}

\newcommand{\walkslennode}[4]{\Swalk^{\; #3}(#1\xrightarrow{#4} #2)}
\newcommand{\walksnode}[3]{\Swalk^*(#1\xrightarrow{#3} #2)}

\newcommand{\kln}[1]{\left(#1\right)^\star}

\title[Semigroup identities of tropical matrices through matrix  ranks]
 {Semigroup identities of tropical matrices \\[2mm]   through matrix  ranks}

\author{Zur Izhakian}

\address{  Institute  of Mathematics,
 University of Aberdeen, AB24 3UE,
Aberdeen,  UK.
    }
    \email{zzur@abdn.ac.uk}

\author{Glenn Merlet}
 \address{Aix Marseille Univ, CNRS, Centrale Marseille, I2M, Marseille, France}
 \email{glenn.merlet@univ-amu.fr}


\date{\today}


\def\Id{\operatorname{Id}}
\def\crt{{\operatorname{cr}}}
\def\crit{{\operatorname G}^{\crt}}
\def\digr{{\operatorname G}}
\def\thrs{{\operatorname T}_{\crt}}
\def\tthrs{{\widetilde{\operatorname T}}_{\crt}}
\def\digrG{{G}}

\def\subcrit{{ H}}
\def\mN{\mathcal{N}}
\def\sr{\lambda}

\begin{document}

\begin{abstract}
We  prove the conjecture that, for any $n$, the monoid of all $n \times n$ tropical matrices satisfies nontrivial semigroup identities.
To this end, we prove that the factor rank of a large enough power of a tropical matrix does not exceed the tropical rank of the original matrix.
\end{abstract}

\maketitle

%%%%%%%%%%%%%%%%%%%%%%%%%%%%%%%%% section %%%%%%%%%%%%%%%%%%%%%%%%%%%%%

%  {\small \tableofcontents}

\section*{Introduction}
\numberwithin{equation}{section}

Tropical matrices are matrices  over the max-plus semiring \cite{pin98}, that is $\Trop\ := \Real\cup\{-\infty\}$ equipped with the operations
of maximum as addition and summation as  multiplication:
$$a \add  b :=\max\{a,b\},\qquad a \mlt b := \operatorname{sum}\{a,b \}.
$$
This semiring is additively idempotent, i.e., $a \add  a = a$ for every $a\in \Trop$, in which
$\zero := -\infty$ is the zero element and $\one := 0$ is the
multiplicative  identity.
More generally, one may consider $\Trop$ as an ordered semiring whose addition is determined as maximum, e.g., a semiring obtained from an ordered monoid $(\tS, \cdot \;)$   by setting the addition to be maximum and $\cdot$ as multiplication.  $\MnT$ denotes the monoid of all $n \times n $ square  matrices  with entries in~$\Trop$, and induced multiplication.
These matrices correspond uniquely to weighted digraphs (see ~\cite{Butkovic,MPatW} for recent expositions), which play a central role  in
 algebraic methods, applications to combinatorics,
semigroup representations, automata theory, and many other methodologies.

Any finitely generated semigroup of tropical matrices has polynomial growth \cite{Ales, Simon}; thus
the free semigroup on two generators is not isomorphic
to a tropical matrix sub-semigroup. Growth rate  of  groups is an important subject of study in  combinatorial and geometric group theory, delivered to semigroup theory as well, involving semigroup identities \cite{SV}.  While Gromov’s theory \cite{Grom} implies that every finitely generated
 group having polynomial growth satisfies a nontrivial semigroup identity (since it is virtually nilpotent),
 Shneerson  has given examples which show that this does not hold for semigroups \cite{Shn}.

Tropical matrices enable natural linear representations of semigroups;  therefore,
the question whether tropical matrices satisfy nontrivial semigroup identities arises immediately \cite{IzMr}.
If they do satisfy identities, then any faithfully represented semigroup inherits these identities, and complicated computations are saved \cite{plc}.
As well, these identities define  varieties of tropically represented semigroups \cite[~Ch.~ VII]{pinSV},
where matrix view may provide a classification (or bases) for these varieties.
 Birkhoff HSP Theorem states that varieties are the only classes of semigroup stable under homomorphisms, submonoids, and products.
In addition, by the one-to-one correspondence,  matrix identities are carried over to labeled weighted digraphs,
with multiplication replaced by walk composition,
and are interpreted as the impossibility of word separation in automata theory \cite{wSEP}. (See Section \ref{ssec:separation} for details.)

Semigroup identities have been found for certain submonoids of tropical matrices, including triangular matrices,
and for arbitrary $2 \times 2 $ and $3 \times 3$ matrices \cite{trID,mxID,IzMr,Shitov,Okninski}.
In this paper we prove the existence of identities for all $n\times n$ tropical matrices, for any  $n$,  to wit:
\thmref{\ref{thm:IdExistence}}{The monoid $\MnT$ satisfies a nontrivial semigroup identity for every $n \in \N$.
The length of this identity grows with~$n$ as~$e^{C n^2+\oo(n^2)}$, for some~$C\le 1/2+\ln (2)$.} 
\noindent
This theorem further supports the insight that, in many senses, the behavior of tropical matrices is similar to that of matrices over a field \cite{AGGu,Iz,IJK,IR1,IR2,IR3,IR4}, and has immediate consequences  in semigroup representations.
%\Com{To rephrase}
\corref{\ref{cor:rep.id}}{Any semigroup which is faithfully represented  by $\MnT$  satisfies a nontrivial identity.}

Our semigroup identities arise from an idea of Y.~Shitov~\cite{Shitov}, resulting in Lemma \ref{l:Shitov}, which paves the way to constructing identities for matrices by induction  on their size.
The further step towards this aim is detecting new relations for those matrices which cannot be factorized to a product of matrices of smaller size,
 said to have  factor rank ~$n$.
Unfortunately, Shitov was only able to deal with matrices having maximal determinantal rank,
%whose rank is maximal with  respect to another type of rank, called determinantal rank,
 and thus to conclude the existence of identities only for $3 \times 3$ matrices.
(See Definition~\ref{def:matOper} for  various notions of rank
and~\cite{ABG} for an extensive  survey.)

To prove Theorem~\ref{thm:IdExistence}, we rely on  tropical rank and give a generalization of the first author's result~\cite{mxID} to obtain identities for matrices of maximal rank (Theorem \ref{thm:FullRank}), based on
identities of  triangular matrices (\cite[Theorem 4.10]{trID} or~\cite{Okninski}).
%This gives us a wider choice of identities than in~\cite{Shitov}.
Since tropical rank is the smallest among other notions of ranks~ \cite{ABG}, especially smaller than determinantal rank,
this is not enough to construct identities for~$\MnT$, and additional ingredient is needed.
Specifying a new  relationship between tropical and factor rank is then a crucial obstacle, confronted in this paper.
We introduce two results of similar flavor.

\propref{\ref{pr:TropToFactor}}{Let $A\in\MnT$ and $\lcmn=\lcm(1,\dots, n)$. If $\trrk(A^\lcmn)<n$,  then $\frk(A^{t\lcmn})<n$ for any  $t\ge 3n-2$.}

\thmref{\ref{thm:frk}}{$\frk(A^{t})\le \trrk(A)$ for any~$A\in\MnT$ and  $t\ge (n-1)^2+1$.}
\noindent
The proof of the latter  is based on the  so-called weak CSR expansion -- a method developed  by  T.~Nowak, S.~Sergeev,
and the second author in~\cite{wCSR}.
The former  is proven in the same spirit, but the simplification derived from the power~$\lcmn$ allows for a self-contained exposition of
graph theoretic arguments.

These results are interesting for their own sake, as they introduce new relationships between different notions of rank, concerning also
their tendency to unite for large powers.
Indeed,  in their earlier paper~\cite{ultRank} the authors have shown that, taking powers of a matrix, at the limit all notions of rank coincide. This limit is reached for irreducible matrices, but the exponent can be arbitrary large.

The paper is organized as follows. Section~\ref{sec:2} recalls the relevant setup and results to be used in the paper.
Section~\ref{sec:3}~introduces the relationships between the factor rank of a matrix power and its original tropical rank.
Section~\ref{sec:4} applies these relationships to prove the existence of semigroup identifies for~ $\MnT$.

%%%%%%%%%%%%%%%%%%%%%%%%%%%%%%%%% section %%%%%%%%%%%%%%%%%%%%%%%%%%%%%
\section{Preliminaries}\label{sec:2} As the paper combines several areas of study, 
 we provide   the relevant background.
 
\subsection{Semigroup identities}

Given an \hlt{alphabet} $\varX$, i.e., a finite set   of \hlt{letters},
the free monoid of finite sequences generated by $\varX$ is denoted by $\varX^*$.
The elements of $\varX^*$ are termed \hlt{words}, its identity element is the empty word, denoted by $e$. The \hlt{length} of a word $w$, denoted by $\wlen(w)$, is the number of its letters.
 We write $\lno{\a}{w}$ for the number of occurrences of a letter $\a$ in $w$.
  Both $\wlen(w)$ and $\lno{\a}{w}$  are nonnegative integers.
The free semigroup $\varX^+$ is obtained from $\varX^*$ by excluding the empty word.

A (nontrivial) \hlt{semigroup identity} is a formal equality $u = v$, written as pair $\sid{u}{v}$, where
$u$ and~$v$ are two  different words in  $\varX^+$, cf. \cite{SV}. For a monoid
identity one  allows $u$ and $v$ to be the empty word as well, i.e.,  $u,v \in \varX^*$.
The \hlt{length} of $\sid{u}{v}$ is defined to be  $\max\{ \wlen(u), \wlen(v)\}$.
An identity  $\sid{u}{v}$ is said to be an \hlt{n-letter identity}, if $u$ and~$v$ involve at most~$n$ different letters from~$\varX$.

 A semigroup $\tS := (\tS, \cdot \;)$ \hlt{satisfies a
semigroup identity}~$\sid{u}{v}$,   if
\begin{equation}\label{eq:s.id}
\text{ $ \phi(u)= \phi(v)$ \ for every semigroup homomorphism $\phi
:\varX^+ \longrightarrow \tS$.}
\end{equation}
The set of all semigroup identities satisfied by $\tS$ is denoted by $\Id(\tS)$.
Note that even if $\tS$ is a monoid or a group, $u,v$ are still taken to be elements of the free semigroup $\varX^+$.
With this setting, $\sid{ab}{e}$ is not a legal semigroup identity, but it is a monoid identity.

\begin{theorem}[{\cite[{Theorem 3.10}]{trID}}]\label{thm:2id} A semigroup that satisfies an
 $n$-letter identity, $n \geq 2$,  also satisfies  a
$2$-letter identity of the same length.
\end{theorem}

In this view, regarding existence of semigroup identities, one may restrict to  a 2-letter alphabet.
Therefore, in the sequel, \hlt{we always assume that $\varX = \{ a,b \}$}.

\begin{notation} Given a word $w \in  \varX^+$ and elements $s',s'' \in \tS$,  we write
  $w\eval{s'}{s''}$ for the evaluation of $w$ in $\tS$, obtained by substituting
   $a \mapsto s'$, $b \mapsto s''$. Similarly, we write
  $\sid{u}{v}\eval{s'}{s''}$ for the pair of evaluations  $u\eval{s'}{s''}$ and $v\eval{s'}{s''}$  in $\tS$ of the words $u$ and $v$.

 In the certain case that $\tS = \varX^+$, to indicate that for $u,v \in \varX^+$ the evaluation  $w\eval{u}{v}$ is again a word in~ $\varX^+$, we use the particular notation $w\substit{u}{v}$. Similarly, we write
  $\sid{u}{v}\substit{u}{v}$ for $\sid{u}{v}\eval{u}{v}$.

\end{notation}

With these notations, condition \eqref{eq:s.id} reads as
\begin{equation*}\label{eq:s.id.2}
\text{ $\sid{u}{v} \in \Id(\tS)$ \ iff \ $u\eval{s'}{s''} = v\eval{s'}{s''}$ for every $s',s'' \in \tS$.}
\end{equation*}
Note also that, if $\sid{u}{v} \in \Id(\tS)$, then $\sid{u}{v}\substit{w_1}{w_2} \in \Id(\tS)$ for any $w_1,w_2 \in \varX^+$.
\subsection{Tropical matrices}\label{sec:mathGr}

Tropical matrices are matrices with entries in $\Trop := \Real \cup \{ -\infty\}$, whose multiplication is
induced from the semiring operations of $\Trop$ as in the familiar matrix construction. The set of all $\nxn$ tropical matrices form the multiplicative monoid  $\Mn := \MnT$. The {identity}  of~
 $\Mn$,  denoted  by $I$,  is the matrix
with $\one := 0$ on the main diagonal and whose off-diagonal
entries are all $\zero := \minf$.
Formally,  for any nonzero matrix $A \in \Mn$ we set $A^{0} :=
I$. A  matrix $A \in  \Mn$ with entries $A_{i,j}$ is written
as $A = (A_{i,j})$, where  $i,j = 1,\dots,n$. We denote by $\Un:= \UnT$ the submonoid of all (upper) tropical triangular  matrices in $\Mn$. We write $\Mat_{m,n} := \Mat_{m,n}(\Trop)$ for the set of all $m \times n$ tropical matrices.
A \hlt{permutation matrix} is an $\nxn$ matrix $P_\pi = (P_{i,j})$, with $\pi$ a permutation over $\{1,\dots,n\}$, such that $P_{i, \pi(i)} = \one$ for each $i = 1, \dots, n$ and $P_{i,j} = \zero $ for all $ j \neq \pi(i)$.

\begin{definition}\label{def:matOper} Given  a tropical matrix $A\in \Mn$.
\begin{enumerate} \eroman \dispace
 \item The \hlt{permanent} of  $A$ is
defined as:
  \begin{equation*}\label{eq:tropicalDet}
 \per(A)= \Add_{\pi \in \Sn}    \Mlt_i A_{i,\pi(i)},
\end{equation*}
where $\Sn$ denotes the set of all the permutations over $\{1,\dots,n\}$.
The \hlt{weight} of a permutation $\pi \in \Sn$ is~$\w(\pi)=\Mlt_i A_{i,\pi(i)}$, so that~$\per(A) = \Add_{\pi \in \Sn}  \w(\pi)$. %\Com{do we need this ? }

 \item $A $ is called \hlt{nonsingular},
 if there exists a unique permutation $\tau_A \in \Sn $ that reaches $\per(A)$; that is,
$ \per(A)  =  \w(\tau_A) = \Mlt_i A_{i,\tau_A(i)} \;.$
Otherwise, $A$ is said to be \hlt{singular}.

\item The \hlt{tropical rank}  of $A$, denoted $\trrk(A)$,  is the largest $k$ for which $A$ has a $k \times k$ nonsingular submatrix.
Equivalently,  $\trrk(A)$ is the  maximal number of independent columns (or rows) of~$A$ for an adequate notion of independence \cite{IR1}.

\item The \hlt{factor rank} (also called Schein/Barvinok rank)  of $A$, denoted $\frk(A)$, is the smallest  $k$
for which $A$ can be written as $A=BC$ with~$B\in \Mat_{n,k}$ and~$C\in \Mat_{k,n}$.
Equivalently, $\frk(A)$ is the minimal number of vectors whose tropical span contains the span of the columns (or rows) of~$A$,
or the minimal number of rank-one  matrices~$A_i$ needed to write $A$ additively as~$A=\Add_i A_i$,~cf.
~\cite{ABG}.

\item The \hlt{trace} $\mtr(A) = \sum_{i} A_{i,i}$ is  the usual trace taken with respect to summation,  although it corresponds to the tropical product of diagonal entries in~$\Trop$.

\end{enumerate}
\end{definition}
\noindent
By definition of~$\trrk$, a matrix $A\in \Mn$ is nonsingular iff $\trrk(A) = n$.
 From the last characterization of ~$\frk$ it readily follows that this rank is subadditive:
\begin{equation}\label{eq:subadd}
 \frk(A\add B)\le  \frk(A)+\frk(B).
\end{equation}
As known, the above notions of rank do not coincide~\cite[~\S8]{AGGu}. Nevertheless, the inequality
\begin{equation}\label{rank.relation1}
 \trrk(A) \leq \frk(A)
\end{equation}
holds for every  $ A \in \Mn$  \cite[Theorem 1.4]{DSSt}.
\pSkip

It is easily seen that
 $\per(A) \geq \mtr(A)$ and $\mtr(AB) \geq \mtr(A) + \mtr(B)$
 for any $A,B \in \Mn$. Furthermore, for products of matrices,  we have the following.

\begin{theorem}[{\cite[Theorem 2.6]{Iz}, \cite[Theorem 3.5]{IR2},  \cite[Proposition 3.4]{merl10}}]\label{pr:perAB}
Any~$A,B \in \Mn$ satisfy
 $$\per(AB) \geq  \per(A) + \per(B).$$
 If $AB$  is nonsingular, then $A$ and~$B$ are nonsingular,
$ \per(AB) = \per(A) + \per(B),$ and  $\tau_{AB}= \tau_B \circ \tau_A. $
\end{theorem}

\subsection{Digraphs and automata}

Any matrix $(A_{i,j}) \in \Mn$ is uniquely associated with the \hlt{weighted digraph}
$\digr(A) := (\ver, \arc)$ over the set of node
$\ver :=\{ \vx{1}, \dots, \vx{n}\}$ with   a directed arc
$\e_{i,j} := (\vx{i}, \vx{j}) \in \arc$ of \hlt{weight}~ $A_{i,j}$ from $\vx{i}$ to $\vx{j}$  for every
$A_{i,j} \ne \zero$.  With this one-to-one correspondence, we say that $\digr(A)$ is the graph of the matrix $A$, and conversely that $B$ is the matrix of the weighted digraph $\digr'$, if $\digr' = \digr(B)$.

A \hlt{walk} $\gm$ on~$\digr(A)$ is a sequence of arcs
$\e_{i_1, j_1}, \dots, \e_{i_m, j_m} $, with $j_{k} = i_{k+1}$ for every
$k = 1,\dots, m-1$. We write $\gm := \gm_{i,j}$ to indicate that
$\gm$ is a walk from $\vx{i} = \vx{i_1}$ to $\vx{j}=\vx{j_m}$.
The \hlt{length}  of a walk  $\gm$, denoted by $\len{\gm}$,  is the number of
its arcs. % We say the $\gm$ is an $\ell$-path if it is of length~$\ell$.
Formally, we may consider also walks of length $0$, one on each node.
The \hlt{weight} of
$\gm$, denoted by $\w(\gm)$,  is the sum of weights of its arcs, counting repeated arcs.

We write $ \gm_{i,j}\circ \gm_{j,h}$ for the composition  of the walk $\gm_{i,j}$ from~$i$ to~$j$ with the walk  $\gm_{j,h}$ from~$j$ to~$h$. Similarly, $(\rho)^k$ denotes for the composition $\rho \circ \dots \circ \rho$ of a loop $\rho$ repeated $k$ times.
A walk~$\gm=\e_{i_1, i_2}, \dots, \e_{i_m, j_m}$
may also be viewed as the sequence of nodes~$(i_1,i_2, \dots, i_m,j_m)$. For convenience, we use this point of view as well,
depending of context, and write $\vx{i}\in\gm$ to indicate that the node~$i$ appears in~$\gm$.

A walk~$\gm$ is \hlt{simple} (or elementary) if it has no repeated nodes, i.e., a node appears in~$\gm$ at most once, except possibly as first and last  node.
A (simple) walk that starts and ends at the same node is called a
\hlt{(simple) cycle}. An arc $\rho_i := \e_{i,i}$ is called a \hlt{loop}.
A \hlt{$1$-cyclic walk} is a walk which contains simple cycles of length at most $1$.

A digraph $\digrG$ is called \hlt{strongly connected}, if there is a walk from~$i$ to~$j$ for any nodes~$i,j$.
Maximal strongly connected subgraphs of~$\digrG$ are called \hlt{strongly connected components} (\hlt{s.c.c.'s}).
When there are  no arcs between different s.c.c.'s, $\digrG$ is said to be \hlt{completely reducible}.

The \hlt{cyclicity}~$\cyclty(\digrG)$ of a strongly connected digraph~$\digrG$ is the greatest common divisor of the lengths of its cycles.
If~$\digrG$ is not strongly connected, then its cyclicity~$\cyclty(\digrG)$ is
the least common multiple of the cyclicities of its s.c.c.'s.
It is well-known that the lengths of all walks on~$\digrG$ which start at a same node~$i$ and end at a same  node~$j$ are congruent modulo~$\cyclty(\digrG)$.

 \begin{remark}\label{per.digr}
 A permutation $\pi \in \Sn$ uniquely corresponds to a disjoint union $\Cycles$ of simple cycles $\cyc_1, \dots, \cyc_m$ that cover all the nodes of $\digr(A)$.
 We say that $\cyc_t$ is a cycle of $\pi$, and write $\w(\Cycles) = \sum_t \w(\cyc_t)$, so that $\w(\Cycles) = \w(\pi)$.
 The permanent of $A$ is the highest weight  $\w(\Cycles)$ over all such   $\Cycles$. Accordingly, a matrix $A \in\Mn$ is nonsingular, if $\digr(A)$ has a unique covering $\Cycles$ of highest weight  by simple cycles.
 \end{remark}

The \hlt{spectral radius} of a matrix $(A_{i,j})\in \Mn$ is the value
\begin{equation}\label{eq:sr}
\sr(A) = \bigvee_{ j\leq n} \bigvee_{i_1,\dots, i_j } \frac{ A_{i_1 i_2} +
  A_{i_2i_3} +\cdots + A_{i_j i_1}}{j},
\end{equation}
that is,  the maximal mean weight of  (simple) cycles in $\digr(A)$.
A simple cycle of $\digr(A)$ is called \hlt{critical}, if its mean weight equals
$\sr(A)$. A node of $\digr(A)$ is said to be a \hlt{critical node}, if it belongs to some critical cycle.
The \hlt{critical graph} of $A$, denoted by $\crit(A)$, is the union of
all the critical cycles of $\digr(A)$, over the node set $\ver$.
If $\digr(A)$ is acyclic, then $\sr(A)=\zero$ and $\crit(A)$ has no arcs. We may also view $\crit(A)$ as a subgraph of $\digr(A)$ over a subset of nodes in $\ver$.

\pSkip

The \textbf{kleene star} of $A$ is the matrix
$$ A^\star =    \Add_{k \in \N}  A^k.$$
(Some entries might be $+\infty$, unless $A$ is normalized by $\sr(A)\le\one$.)

\begin{remark}\label{rk:AkWalk}
Taking a power $A^t$ of a matrix $A \in \Mn$, it is easily verified that the entry
$(A^t)_{i,j}$ is the highest weight of walks  from~$i$ to~$j$ of length~$t$ on $\digr(A)$.
Thus, the $(i,j)$-entry of
$A^\star$ is the supremum of the weights of all walks from~$i$ to~$j$ on $\digr(A)$.
Obviously, the supremums are reached and~$A^\star\in\MnT$ if the weight of the cycles are nonpositive, that is if~$\sr(A)\le\one$.
\end{remark}

While powers of a single matrix correspond to walks on $\digr(A)$, to deal with products of  matrices, $\digr(A)$ needs a generalization.
We restrict to products of two matrices, which suffices our purpose.

\begin{definition}\label{def:automata}
The \hlt{labeled-weighted digraph} $\Auto(A,B)$, written \hlt{lw-digraph}, of matrices $A,B\in\Mn$ is the digraph over the nodes~$\ver :=\{ 1,\dots, n\}$ with a directed arc $\e_{i,j}$
from $i$ to $j$ labeled~$a$ of weight~ $A_{i,j}$   for every
$A_{i,j} \ne \zero$ and a directed arc from $i$ to $j$ labeled $b$ of weight $B_{i,j}$  for every
$B_{i,j} \ne \zero$.
A walk~$\gm=\e_{i_1, j_1}, \dots, \e_{i_m, j_m}$ on~$\Auto(A,B)$ is labeled by the sequence of arcs' labels  along $\gm$, from~$\e_{i_1, j_1}$ to~$\e_{i_m, j_m}$,
which is a word in~$\{a,b\}^+$. (In particular, every walk labeled by~$w$ has length~$\wlen(w)$.) The weight $\w(\gm)$ of $\gm$ is the sum of its arcs' weights.
\end{definition}
\noindent Note that $\Auto(A,B)$ may have parallel arcs, but with different labels, and that $\Auto(A,A) = \digr(A)$.
With this definition, we have the following proposition.
\begin{proposition}\label{pr:WalkInterpret}
  Given a word $w \in \{ a,b\}^+$ of length $\wlen(w)$ and matrices $A,B \in \Mn$,
  the $(i,j)$-entry of the matrix $w\eval{A}{B}$ is the maximum over the weights of all walks $\dl_{i,j}$ on~$\Auto(A,B)$ from $i$ to $j$ labeled by~$w$.
\end{proposition}

lw-digraphs are the core of  \hlt{weighted automata} -- a widely studied extension of standard  (i.e., boolean)  automata.
(See \cite{HBookAutom} for an overview on automata theory.)

\begin{remark}\label{rem:WalkInterpret}
In automata theory, nodes are called states, arcs are called transitions, and one consider set of walks (also called runs) from an initial to a final state.
The initial and final state might also have weights. Thus, a weighted automaton is defined by $(\Auto(A,B),\bf{r},\bf{c})$, where $\bf{r}$ is a row vector and~$\bf{c}$ is a column vector.
The weight of a word $w$ is the sum (here max) of weights of walks labeled by $w$, given by ${\bf r}\left(w\eval{A}{B}\right){\bf c}$.
\end{remark}

\subsection{Word separation}\label{ssec:separation}
Due to Remark~\ref{rem:WalkInterpret}, existence of a nontrivial semigroup identity for ~$\Mn$ can be understood as the impossibility to~\hlt{separate} two words by  weighted automata.
Recall that a standard  automaton is said to separate a pair of words~$(u,v)$ if it accepts~$u$ but not~$v$.
Determining the size of the  smallest automaton that separates a pair of words is an old open problem in automata theory. See~\cite{wSEP} for a survey on the subject.
Let $\sep(m)$ be the smallest size of automata, in terms of state number $n$, necessary to separate all pairs of words of length~$m$ (or at most~$m$, since separating words of  different length is easier).
The best known upper bound is $\sep(m)=\OO\big(m^{2/5}\ln^{3/5}(m)\big)$ \cite[Theorem~3]{Robson89}.

Since there are finitely many automata having~$n = \sep(m)$ states, and~$2^m$ words of length~$m$ in $\{a,b \}^*$, obviously, there is a pair of words that cannot be separated by such an automaton.
This simple argument on cardinality gives words of length~$2^{2n^2+\oo(n)}$. This means that there exists  a nontrivial semigroup identity of length~$2^{2n^2+\oo(n)}$, satisfied by the monoid $\Mn$ of $\nxn$ boolean matrices.
Analyzing  powers of boolean matrices, shorter identities of order~$e^{n+\oo(n)}$ are obtained, so that~$\sep(m)\ge \ln(m)+\oo(\ln (m))$. To the best of our knowledge, this is the best lower bound.

A weighted automaton is said to separate two words, if it assigns these words with different weights.
As there are infinitely many weighted automata having a given number of states, it is not obvious that
not all pairs of words can be separated by automata of a given size.
Denote  by~$\sep_{\Sem}(m)$ the smallest size of weighted automata, having weights in the semigroup~$\Sem$,   necessary to separate all pairs of words of length~$m$ (or at most~$m$).
It follows from Remark~\ref{rem:WalkInterpret} that $\sep_\Sem(m)>n$ iff there exists a semigroup identity for~$\Mn(\Sem)$.
Theorem~\ref{thm:IdExistence} implies that $\sep_\Trop(m)\ge c \ln^{1/2}(m)$ for some~$c>0$.
As far as we know, there is no better upper bound than the one for boolean matrices.

\section{Ranks of large powers of a matrix}\label{sec:3}

In  this section we assume that $A$ is a matrix in $\Mn$, and set $\lcmn=\lcm(1,\dots,n)$.
\subsection{Direct approach}

\begin{lemma}\label{lem:permAn} % Given $A_n \in  \Mn$,  set   $N=\lcm(1,\dots,n)$.
  If  $\trrk(A^\lcmn)=n$, then~$\per(A^\lcmn)=\mtr(A^\lcmn)$.
\end{lemma}
\begin{proof}
Follows from Theorem~\ref{pr:perAB}. (See also~\cite[Lemma 2.8]{mxID} or~\cite[Corollary 4]{Shitov}.)
\end{proof}
We start with an  easy lemma that links weights of permutations to weights of simple cycles.

\begin{lemma}\label{l:cycToPerm}
Given  a permutation $\tau \in \Sn$, let~$\mu_i$ be the average weight of the unique (simple) cycle  of~$\tau$ that  contains the node~$i$. Suppose
\begin{equation}\label{eq:CycToTau}
 \w(\cyc)<\sum_{i\in\cyc} \mu_i,
\end{equation}
for  every simple cycle~$\cyc$ of $\digr(A)$ which is not a cycle of $\tau$,
then $A$ is nonsingular and~$\tau=\tau_A$. \end{lemma}

\begin{proof}
A permutation~$\pi \in\Sn$  corresponds  to a disjoint union of simple cycles~$\cyc_1,\dots, \cyc_m$ (cf. Remark \ref{per.digr}).
If~$\cyc_t$ is also a cycle of~$\tau$, then  $\mu_i=\frac{\w(\cyc_t)}{\len{\cyc_t}}$ for every~$i\in\cyc_t$, by definition of~$\mu_i$. Therefore $\w(\cyc_t)=\sum_{i\in\cyc_t}\mu_i$.
Using~\eqref{eq:CycToTau}, we see that
$$\w(\pi)=\sum_{t=1}^m\w(\cyc_t)\le\sum_{t=1}^m\sum_{i\in\cyc_t}\mu_i =\sum_{i=1}^n \mu_i=\w(\tau),$$
where  equality can only be reached if all cycles of~$\pi$ are cycles of~$\tau$; that is, if~$\pi=\tau.$
This implies that $\tau$ is maximally  unique;  hence, $A$ is nonsingular  and~$\tau=\tau_A$.
\end{proof}

\begin{remark}\label{rem:crit}
%Let $N = \lcm(1,\dots, n)$, and take   $A \in \Mn$.
 The critical nodes of $\crit(A^\lcmn)$  and ~$\crit(A)$ are the same, while $\crit(A^\lcmn)$ also has a loop at each node belonging to a critical cycle of~$A$.
\end{remark}

More generally:
\begin{lemma}[{\cite[Lemma 3.6]{MPatW}}]\label{l:CAk}
The matrix  of $\crit(A^{t})$ is the $t$'th power of the matrix  of  $\crit(A)$, for any % $A\in\Mn$,
$t\geq 1$. Moreover, $\sr(A^t)=t\sr(A)$.
\end{lemma}

From this simple lemma, in the spirit of~\cite{wCSR}, we deduce:
\begin{lemma}\label{l:2simple}
Suppose that  $B=A^\lcmn$ for some~$A\in\Mn$, % where $\lcmn=\lcm(1,\dots, n)$,
and that $t\ge 2n-2$. For every $(B^t)_{i,j} \neq \zero$   there exists a walk $\gm_{i,j}$ from $\vx{i}$ to $\vx{j}$ on~$\digr(B)$, with  weight $(B^t)_{i,j}$ and  length~$t$,   of the form
\begin{equation}\label{eq:2simple}
\gm_{i,j} = \gm_{i,h} \circ \rho^{s} \circ \gm _{h,j},
\end{equation}
where $\rho$ is a loop at node~$\vx{h}$, and $\gm_{i,h}$, $\gm_{h,j}$ are simple walks, possibly empty.\footnote{Note that  $\gm_{i,j}$ needs not be 1-cyclic.}

\end{lemma}
\begin{proof}
Proof by induction on the matrix size $n$.
The case of~$n=1$ is trivial.

Assume $n >1 $.
 % Write $\gm := \gm_{i,j}$.
Since $B^t=A^{t \lcmn}$, for $(B^t)_{i,j} \neq \zero$  there is a walk $\gm_0 $   on~$\digr(A)$ from~$\vx{i}$ to~$\vx{j}$ of length~$t \lcmn$ and weight $(B^t)_{i,j}$.
Note that $\digr(A)$ and $\digr(B)$ have the same critical nodes (Remark~\ref{rem:crit}), so  we  often say critical node without specific details.
There are two cases.
\pSkip
\noindent
\textbf{Case I:} $\gm_0$ passes through a critical node $\vx{c}$.

\noindent
Let $\cyc$ be a simple cycle, possibly a loop, on~$\crit(A)$ which contains $\vx{c}$.
We insert $\cyc$ repeated $\lcmn/\len{\cyc}$ times in~$\gm_0$ to obtain a walk~$\gm_1$ on $\digr(A)$.
This walk has  length~$(t+1)\lcmn$ and weight~$B_{i,j}+\lcmn\sr(A)$, where $\sr(A)$ is the spectral radius of $\digr(A)$ defined by \eqref{eq:sr}.

 The sequence of nodes positioned at $0, \lcmn,  2\lcmn, \dots, (t+1)\lcmn  $ in ~$\gm_1$ determines a walk~$\gm_2$ on~$\digr(B)$ that passes  through a critical node $\vx{c'}$ (not necessarily $\vx{c}$). This walk  has length~$t+1$ and weight at least $B_{i,j}+\lcmn\sr(A)=B_{i,j}+\sr(B)$.

\begin{itemize}[leftmargin=.2in]
  \item[--] If $\vx{c'}$ appears twice (or more) in $\gm_2$, then the closed subwalk from the first occurrence of $\vx{c'}$ in $\gm_2$ to last occurrence can be replaced by loops on $\vx{c'}$.
  \item[--] If a node $\vx{k}$  appears twice (or more) on the same side of the occurrence of~$\vx{c'}$ in~$\gm_2$, then the closed subwalk from the first occurrence of $k$ to last  occurrence can be replaced by loops on~$\vx{c'}$.
\end{itemize}
These exchanges do not decrease the weight of $\gm_2$, since  the loop at $c'$ is critical.
They provide a new walk~$\gm_3$ which can be decomposed as in~\eqref{eq:2simple}, but has length~$t+1$ and weight at least $(B^t)_{i,j}+\sr(B)=(B^t)_{i,j}+B_{c,c}$.
Since $t+1\ge 2(n-1)+1$, while simple walks have length at most~$n-1$, $\gm_3$ has  at least one loop~$\rho$ which can be removed to get the desired walk~$\gm_4$ of  the right length.
The weight of~$\gm_4$ is proved to be at least~$(B^t)_{i,j}$, but, clearly, it cannot be strictly greater.

\pSkip
\textbf{Case II:} $\gm_0$ does not pass through any critical node.

\noindent
Then, $\gm_0$ is a walk on the graph $\digr(\tlA)$, where $\tlA$ is the matrix  obtained  by deleting  from~$A$ the rows and columns corresponding to critical nodes.
Thus, $(B^t)_{i,j}=(\tlA^{ t \lcmn })_{i,j}$ and, by the induction hypothesis,
$(B^t)_{i,j}$ is the weight of a walk~$\gm_{i,j}$ of length~$t$ on~$\digr(\tlA^\lcmn)$ of the form~\eqref{eq:2simple}.
By definition of~$\tlA$ and Remark~\ref{rk:AkWalk}
the weights of the arcs of~$\digr(\tlA^\lcmn)$ are at most the weights of the corresponding arcs of~$\digr(B)$,
so that the weight of~$\gm_{i,j}$ as a walk on~$\digr(B)$ is at least~$(B^t)_{i,j}$. As it cannot be strictly greater,
$\gm_{i,j}$ has weight~$(B^t)_{i,j}$.
\end{proof}

\begin{proposition}\label{pr:TropToFactor} %Let $A\in\Mn$, and  let  $\lcmn=\lcm(1,\dots, n)$.
If $\trrk(A^\lcmn)<n$, then $\frk(A^{t\lcmn})<n$ for any $t\ge 3n-2$.
\end{proposition}
%\Com{To change for the improved version with rank, if the proof works}
\begin{proof} Set $B=A^\lcmn$ and $t\ge 3n-2$.
Assume that $\trrk(B)<n$. Since $B$ is singular,
by Lemma~\ref{l:cycToPerm} applied to the identity permutation, there is  a simple cycle~$\cyc$ on~$\digr(B)$, which is not a loop,
whose weight is at least the sum of weights of loops at  its nodes.
Let $\vx{c}$ be a node of~$\cyc$ whose loop has minimal weight, and let $\vx{h_0}$ be the node proceeding  ~$c$ in~$\cyc$.
We prove that for any~$i,j$:
\begin{equation}\label{eq:ineqNoc}
 (B^t)_{i,j}\le \Add_{h\neq c}\left((B^{n})_{i,h}+(B^{t-n})_{h,j}\right).
\end{equation}

If~$(B^t)_{i,j} = \zero$, then this inequality holds trivially.
Otherwise, let~$\gm_{i,j}=\gm_{i,h}\circ\rho^s\circ\gm_{h,j}$ be a walk given by Lemma~\ref{l:2simple}, where~$h$ is a node of~$\gm_{i,j}$.
Since~$\gm_{i,h}$ and~$\gm_{h,j}$ are simple walks, they have length at most~$n-1$, so that all nodes of~$\gm_{i,j}$
at positions~$n$ to~$t-n+1$ are the same, namely~$h$.
\begin{itemize}[leftmargin=.2in]
  \item[--]
When~$h\neq c$, \eqref{eq:ineqNoc} follows from
$$ (B^t)_{i,j}=\w(\gm_{i,j})
=\w\left(\gm_{i,h}\circ\rho^{n-\len{\gm_{i,h}}}\right)+\w\left(\rho^{t-n-\len{\gm_{h,j}}}\circ\gm_{h,j}\right)\le (B^{n})_{i,h}+(B^{t-n})_{h,j}.$$

\item[--]
If~$h=c$, then $\rho$ is the loop at~$c$. We have $n-1 +\len{\gm_{h,j}}+\len{\cyc}\le 3n-2\le t$ and, by definition of~$c$,  $\w(\cyc)\ge\len{\cyc}\w(\rho)$, so that
$$\begin{array}{ll}
 (B^t)_{i,j}=\w(\gm_{i,j})
&=\w\left(\gm_{i,h}\circ\rho^{n-1-\len{\gm_{i,h}}}\right)+\w\left(\rho^{\len{\cyc}}\right)+\w\left(\rho^{t-(n-1)-\len{\gm_{h,j}}-\len{\cyc}}\circ\gm_{h,j}\right)\\[2mm]
&\le\w\left(\gm_{i,h}\circ\rho^{n-1-\len{\gm_{i,h}}}\right)+\w\left(\cyc\right)+\w\left(\rho^{t+1-n-\len{\gm_{h,j}}-\len{\cyc}}\circ\gm_{h,j}\right)\\[2mm]
&=\w\left(\gm_{i,h}\circ\rho^{n-1-\len{\gm_{i,h}}}\circ\cyc\circ\rho^{t+1-n-\len{\gm_{h,j}}-\len{\cyc}}\circ\gm_{h,j}\right)\\[2mm]
&\le (B^{n})_{i,h_0}+(B^{t-n})_{h_0,j}.
\end{array}$$
\end{itemize}
Thus, inequality~\eqref{eq:ineqNoc}  holds in all cases.
Since the reverse inequality  always holds, $B^t$ is the tropical sum of the $n-1$
matrices $\big((B^{n})_{i,h}+(B^{t-n})_{h,j}\big)_{i,j}$ with~$h\neq c$.
 Each of these matrices has rank~$1$, as it is the tropical product of a row
of~$B^n$ by a column of~$B^{t-n}$.
Therefore,  Definition~\ref{def:matOper}.(iv)  of factor rank implies $\frk(B^t)<n$.
\end{proof}

\subsection{CSR approach}
To prove Theorem~\ref{thm:frk} below we use the so-called
weak CSR expansion of powers, developed by T.~Nowak, S.~Sergeev and the second author \cite{wCSR}.
We first recall the relevant setup and results.

\begin{definition}\label{def:CSR}
For a completely reducible subgraph~$\subcrit$ of~$\crit(A)$,   $A\in\Mn$,  we
set $$M_\subcrit=\kln{\big(-\lambda(A)+ A\big)^{\cyclty(\subcrit)}}
%\left((-\lambda(A) + A\big)^{\cyclty(\subcrit)}\right)^*
,$$
and define the matrices
$C = C_\subcrit, S = S_\subcrit, R = R_\subcrit$ in $\Mn$ as follows
\begin{equation}
\label{e:csrdef}
C_{i,j}=
\begin{cases}
M_{i,j} &\text{if $j \in \subcrit$,}\\
\zero &\text{otherwise,}
\end{cases}\qquad
S_{i,j} =
\begin{cases}
A_{i,j} &\text{if $(i,j) \in \subcrit$,}\\
\zero &\text{otherwise,}
\end{cases}
\qquad
R_{i,j}=
\begin{cases}
M_{i,j} &\text{if $i \in \subcrit$,}\\
\zero &\text{otherwise.}
\end{cases}
\end{equation}
The matrices $C_\subcrit$, $S_\subcrit$ and~$R_\subcrit$ are named   the \hlt{CSR terms} of~$A$ with respect to~$\subcrit$.
\end{definition}

This CSR expansion provides a useful tool for analyzing tropical matrices, especially their powers.
For this purpose, we are  interested in products ~$C_\subcrit (S_\subcrit)^t R_\subcrit$ with $t \in \N$,
whose interpretation in terms of walks on~$\digr(A)$ is given by Theorem~\ref{thm:representation} below.

\begin{remark}\label{rem:acyclic}
If~$\digr(A)$ is acyclic, then~$\sr(A)=\zero$ and the matrices~$M_H, C_H, S_H, R_H$ are not defined by~\eqref{e:csrdef}. In this case,~$\crit(A)$ has no arcs, and we formally set these matrices to  be zero matrix, which is consistent with Theorems~\ref{thm:wCSR} and~\ref{thm:representation} below.

Alternatively to \eqref{e:csrdef}, when $\sr(A)\neq \zero$, the matrices $C_\subcrit$ and $R_\subcrit$ can be extracted respectively
from the columns and the rows of~ $M_H$ which are indexed by the nodes of $\subcrit$, while $S_\subcrit$ can be obtained from the square submatrix indexed by the critical nodes of~$\crit(A)$.
The products~$C_\subcrit (S_\subcrit)^tR_\subcrit$ obtained with this approach are the same as those obtained via~\eqref{e:csrdef}.
Note  that  $M_\subcrit$, $C_\subcrit$ and~$R_\subcrit$ remain unchanged when (tropically) multiplying $A$ by any  $\al \in \Real$, but $S_\subcrit$ is multiplied by~$\al$.

\end{remark}

The matrix~$B[A]$  is defined by\footnote{It is called the \hlt{Nachtigall matrix subordinate to~$A$} in~\cite{wCSR}, denoted there by $B_N$.} %\Com{unclear}
\begin{equation}\label{e:CSRschemes}
\left(B[A]\right)_{i,j} =
\begin{cases}
\zero & \text{if $i$ or $j$ is a critical node in $\crit(A)$},  \\
A_{i,j} & \text{else}.
\end{cases}
\end{equation}
 In graph view, the digraph $\digr(B[A])$  is the subgraph of~$\digr(A)$ induced by the set of non-critical nodes, i.e., the digraph obtained from $\digr(A)$ by omitting all arcs incident
to  critical nodes, in particular all arcs of $\crit(A)$.
Therefore, if $\digr(A)$ is acyclic, then $B[A] = A$.

\begin{theorem}[{\cite[Theorem~4.1]{wCSR}}]\label{thm:wCSR}
 Given $A \in \Mn$, let $C_\subcrit$, $S_\subcrit$, $R_\subcrit$ be the  CSR terms~\eqref{e:csrdef} of $A$ for~$\subcrit=\crit(A)$,
  and let~$B[A]$ be the matrix~\eqref{e:CSRschemes}.
Then
\begin{equation}\label{e:wCSR}
A^t=C_\subcrit (S_\subcrit)^t R_\subcrit \add \left(B[A]\right)^t, \quad \text{for any~$t\ge \Wi$.}
\end{equation}

\end{theorem}
\noindent
Note that~$\crit(A)$ is a completely reducible subgraph of $\digr(A)$, unless~$\digr(A)$ is acyclic.
In the latter  case, all matrices in \eqref{e:wCSR} are zero, so the equation holds obviously.

A main approach for proving CSR results, for instance Theorem \ref{thm:wCSR}, is the interpretation of a product
$C_H (S_H)^t R_H$ in terms of walks on $\digr(A)$, based on the  following notations ($\mN$ denotes a node subset):
\begin{itemize}%\label{def:WalkSets}
\dispace
 \item $\walkslen{i}{j}{t}$ is the set of all walks from~$i$ to~$j$ of length~$t$,
 \item $\walkslennode{i}{j}{t}{k}=\bigcup_{t_1+t_2=t}
 \left\{\gm_{i,k}\circ\gm_{k,j} \ds| \gm_{i,k}\in\walkslen{i}{k}{t_1}, \; \gm_{k,j}\in\walkslen{k}{j}{t_2} \right\}$,
\item $\walkslennode{i}{j}{t}{\mN}=\bigcup_{k\in\mN}
\walkslennode{i}{j}{t}{k} $,

 \item $\walksnode{i}{j}{\mN} =
\bigcup_{t\geq0} \walkslennode{i}{j}{t}{\mN}$,

 \item $\walkslennode{i}{j}{t,\p}{\mN} =\big\{\gm \in \walksnode{i}{j}{\mN}
\,\big|\,\wlen(\gm) = t \mod \p \big\}$, with $\p \in \N$.
\end{itemize}

\begin{theorem}[{\cite[Theorem~6.1]{wCSR}}]\label{thm:representation}
Let~$A\in\Mn$ be a matrix with~$\sr(A)= \one$, and let $C_\subcrit$, $S_\subcrit$, $R_\subcrit $ be
the CSR terms of~$A$ for $\subcrit$ a completely reducible subgraph of~$\crit(A)$.
Let~$p \in \N$ be a multiple of~$\cyclty(\subcrit)$, and let~$\mN$ be a subset of nodes of~$\subcrit$ that contains
at least one node from each  s.c.c.\ of~$\subcrit$.
Then, for every~$i,j = 1,\dots, n$ and~$t\in\N$:
\begin{equation}\label{e:representation}
\big(C_\subcrit (S_\subcrit)^t R_\subcrit \big)_{i,j}=\max\big \{w(\gm) \ds | \gm\in \walkslennode{i}{j}{t,p}{\mN}\big \}.
\end{equation}
\end{theorem}
%\Com{Give reference}

The theorem has the following corollaries
\begin{corollary}[{\cite[Corollary~6.2]{wCSR}}]\label{c:csr-indep}
$C_\subcrit (S_\subcrit)^t R_\subcrit $ depends only on the set of s.c.c.'s of~$\crit(A)$
intersecting ~$\subcrit$ -- a completely reducible subgraph of~$\crit(A)$.
\end{corollary}

\begin{corollary}[{\cite[Corollary~6.3]{wCSR}}]\label{c:CSRscc}
If $\subcrit_1,\dots,\subcrit_q$ are the s.c.c.'s of~$\subcrit$,
then
\begin{equation}\label{e:early-exp}
C_\subcrit (S_\subcrit)^t R_\subcrit =\Add_{\xi=1}^q C_{\subcrit_\xi} (S_{\subcrit_\xi})^t R_{\subcrit_\xi}.
\end{equation}
\end{corollary}

\begin{definition}\label{def:TcrEp}
Let~$\subcrit$ be a subgraph of~$\digr(A)$, and let $p\in\N$.
The \hlt{cycle removal threshold}~$\thrs^p(\subcrit)$
(resp.  \hlt{strict cycle removal
threshold~$\tthrs^p (\subcrit)$}) of~$\subcrit$ is
the smallest $T \in \N \cup \{ 0 \} $ for which the following
holds: for each walk~$\gm \in\walksnode{i}{j}{\subcrit}$ of
length~$\geq T$ there is a walk
$\dl \in\walksnode{i}{j}{\subcrit}$ obtained from~$\gm$ by removing
cycles (resp.\ at least one cycle), and possibly  inserting cycles
from~$\subcrit$, such that $\wlen(\dl)\le T$ and $\wlen(\dl) =  \wlen(\gm) \mod{p}$.

\end{definition}

\begin{proposition}[{\cite[Proposition 9.5]{wCSR}}]\label{p:TcRLin}
Given a subgraph $\subcrit$ of~$\digr(A)$ with~$m$ nodes, then
$$ \thrs^p(\subcrit)\le p n +n-m -1, \quad \text{for any } p \in \N.$$
\end{proposition}

\begin{corollary}\label{c:TcRLin}
For a simple cycle $\cyc$ of~$\digr(A)$ with~$\len{\cyc}\le n-1$ the following holds:
$$\thrs^{\len{\cyc}}(\cyc)\le \len{\cyc} n +n-\len{\cyc}-1=\len{\cyc}(n-2)+n+ \len{\cyc}-1 \le \Wi+(\len{\cyc}-1).
%(n-1)^2+\len{\cyc}.
$$
\end{corollary}

\begin{corollary}\label{c:TcRLin2}
For a node $i$ of  $\digr(A)$  and~$p \le n$ the following holds: % \Com{Why $n_1 = m  = 2$ ? }
$$\tthrs^{p}(\{ i \})\le p n +n-2+1\le n^2 +(p -1).$$
\end{corollary}

The next proposition allows to deal with Hamiltonian cycles.
\begin{proposition}[{\cite[Proposition 9.4]{wCSR}}]\label{p:TcrHAWielandt}
For  a simple cycle $\cyc$ of length~$n$ in~$\digr(A)$ the following holds:
%\Com{Unclear, simple cycle ? }
$$\tthrs^{n}(\cyc)\le n^2-n+1=\Wi+(\len{\cyc}-1).$$
\end{proposition}

\begin{remark}\label{rem:TcrUse}
The above bounds on~$\thrs^p$ are applied to produce  from a given walk~$\gm_1\in\walkslennode{i}{j}{t}{\subcrit}$ a~ new walk~$\gm_2\in\walkslennode{i}{j}{t,p}{\subcrit}$ with $\len{\gm_2}\le T$,
and the bound $\thrs^{p}(\subcrit)\le T$, by omitting cycles from $\gm_1$ and possibly inserting cycles of~$H$.
\end{remark}

The next lemma (included implicitly in~\cite{wCSR}) completes Theorem~\ref{thm:representation} for matrices $A \in \Mn$ that are not normalized, i.e.,  have $\sr(A) \neq \one$.

%\Com{Normalised? the theorems applies to any matrix}
\begin{lemma}\label{l:NonNormCSR}Given~$A\in\Mn$, $t\in\N$,  and indices~$i,j = 1,\dots, n$.
\begin{enumerate} \eroman
   \item $\left(C_\subcrit (S_\subcrit)^t R_\subcrit\right)_{i,j} \ge \w(\gm_{i,j})$ for any completely reducible  subgraph~$\subcrit$ of~$\crit(A)$ and any walk  $\gm_{i,j}\in~\walkslennode{i}{j}{t}{\subcrit}$;

 \item Assume $\cyc$ is a simple critical cycle. If~$t \ge \Wi$, then there exists $\gm_{i,j}\in~\walkslennode{i}{j}{t}{\cyc}$ such that
 $\left(C_\cyc (S_\cyc)^t R_\cyc\right)_{i,j} =\w(\gm_{i,j})$.
\end{enumerate}
\end{lemma}

\begin{proof}
 Both in~$(i)$ and~$(ii)$, since
 $\left(C_\subcrit (S_\subcrit)^t R_\subcrit\right)_{i,j}$ and $\w(\gm_{i,j})$  decreases by~$t\sr(A)$ when
 replacing $A$ by~$(-\sr(A))+  A$,  we may assume  that~$\sr(A)=\one$.
\pSkip
$(i)$: Follows immediately from Theorem~\ref{thm:representation}.
\pSkip
$(ii)$:
 Take $\gm_1\in \walkslennode{i}{j}{t,g}{\mN}$ having  weight~$\left(C_\cyc (S_\cyc)^t R_\cyc\right)_{i,j}$, cf. Theorem~\ref{thm:representation}.
 If $\len{\cyc}=n$, apply Proposition~\ref{p:TcrHAWielandt}, otherwise use Corollary~\ref{c:TcRLin},
 to get a walk~$\gm_2$ such that  $\len{\gm_2} = t \mod{\len{\cyc}}$
 and $\len{\gm_2}\le t+\len{\cyc}-1$, so $\len{\gm_2}\le t$.
 If needed, insert additional copies of~$\cyc$ to get a walk~$\gm_3$ of length exactly~$t$.
 Since~$\sr(A)=\one$, $\w(\gm_2)\ge\w(\gm_1)$.
 Since $\cyc$ is critical, $\w(\gm_3)=\w(\gm_2)$.
 Thus, $\w(\gm_3)\ge\w(\gm_1)=\left(C_\cyc (S_\cyc)^tR_\cyc\right)_{i,j}$.
  The reverse inequality is given  by~$(i)$,
 so that $\gm_{i,j}=\gm_3$ has the desired properties.
\end{proof}

We are now ready to prove the main result of this section.
\begin{theorem}\label{thm:frk} $\frk(A^{t})\le \trrk(A)$ for  any $t\ge (n-1)^2+1$.
\end{theorem}

\begin{proof}
Fix $t\ge\Wi$, and apply  Theorem~\ref{thm:wCSR} recursively to get $A^t$ as the sum of $C_{\subcrit_\xi}(S_{\subcrit_\xi})^t R_{\subcrit_\xi}$ defined by successive matrices subordinate to~$A$.
Explicitly, we start with $A_1=A$ and define inductively  $A_{\xi+1}=B[A_\xi]$.
At each step, we set $C_\xi = C_{\subcrit_\xi}$, $S_\xi = S_{\subcrit_\xi}$, $R_\xi= R_{\subcrit_\xi}$
to be the $CSR$ terms of~$A_\xi$ with respect to~$\subcrit_\xi = \crit(A_\xi)$.

By definition~\eqref{e:CSRschemes} of $B[A_\xi]$ we get a sequence of nested digraphs
\begin{equation}\label{eq:dig.seq} \digr(A) =
 \digr(A_1) \supseteq \digr(A_2) \supseteq \digr(A_3) \supseteq \cdots,
\end{equation}
such that for any~$\xi>\varsigma$, $\digr(A_{\xi})$ is the subgraph of~$\digr(A_\varsigma)$ obtained
by removing all the arcs of $\digr(A_\varsigma)$
that are incident to some node that is critical for some~$A_\zeta$, where  $\xi>\zeta\ge\varsigma$.
Thus, $\digr(A_\xi)$  can be viewed as a digraph on a subset of nodes of $\digr(A)$, i.e., as an induced subgraph.

Since $\digr(A_1)$ has finitely many nodes and $\crit(A_\xi)$ and $\crit(A_\varsigma)$ are arc-disjoint  for any $\xi\neq\varsigma$, the sequence
\eqref{eq:dig.seq} stabilizes after finitely many steps, when~$\digr(A_\xi)$ is acyclic.
Therefore, \eqref{eq:dig.seq} restricts to matrices  $A_1, \dots, A_q$ with  strict inclusions, where $A_q^t=\zero$, since~$\digr(A_q)$ is acyclic.
Applying Theorem~ \ref{thm:wCSR} recursively, we obtain
$$ A^t=\Add_{\xi=1}^{q-1} C_{\xi}(S_{\xi})^tR_{\xi}.$$

If~$\xi<q$, then $\subcrit_\xi$ is not acyclic; hence $\subcrit_\xi$ is completely reducible.
Let~$\Cycles_\xi$ be a collection of simple cycles of~$\subcrit_\xi$ that contains one cycle from each~s.c.c. of~$\subcrit_\xi$, each of them having  minimal length.
By Corollaries~~\ref{c:csr-indep} and \ref{c:CSRscc}, we have
$$C_{\xi}(S_{\xi})^t R_{\xi}=\Add_{\cyc\in\Cycles_\xi}C_{\cyc}(S_{\cyc})^t R_{\cyc},$$
where $C_{\cyc}$, $S_{\cyc}$, $R_{\cyc}$ are CSR terms of $A_\xi$ with respect to~$\cyc \in \Cycles _\xi$.
Namely, the collection $\Cycles=\bigcup_\xi\Cycles_\xi$ of node-disjoint simple cycles gives
\begin{equation}\label{e:CSRExp}
 A^t=\Add_{\cyc\in\Cycles} C_{\cyc}(S_{\cyc})^t R_{\cyc},
\end{equation}
where $C_{\cyc},S_{\cyc},R_{\cyc}$ are the CSR terms of the unique~$A_\xi$ such that~$\cyc$ is a simple cycle
of~$\crit(A_\xi)$.
The factor rank is subadditive, cf.~\eqref{eq:subadd},  and thus \eqref{e:CSRExp} implies
\begin{equation} \label{e:frk&Cycles}
 \frk(A^t)\le \sum_{\cyc\in\Cycles}\frk\left(C_{\cyc}(S_{\cyc})^t R_{\cyc}\right)
\le \sum_{\cyc\in\Cycles} \len{\cyc}.
\end{equation}
(Later we show that some terms can be omitted to get
$\sum_{\cyc\in\Cycles} \len{\cyc}\le\trrk(A)$.)

Let $\SCycles$ be a subcollection of~$\Cycles$ for which ~\eqref{e:CSRExp} holds as well. \hlt{Assume that}
\begin{equation}\label{eq:assume}
\sum_{\cyc\in\SCycles} \len{\cyc}>\trrk(A).
\end{equation}
Denote by~$\tN$ the set of  nodes of all~$\cyc\in\SCycles$, and by~$Q$ the principal minor of~$A$ indexed by the nodes in~$\tN$.
Since the $\cyc\in\SCycles$ are node-disjoint simple cycles, $\tN$ has strictly more than~$\trrk(A)$ elements; hence $Q$ is singular.
By Lemma~\ref{l:cycToPerm}, applied to the permutation of $\tN$ whose cycles are the $\cyc\in\SCycles$ (cf. Remark \ref{per.digr}),
$\digr(Q)$ has a simple cycle~$\cyc\notin\SCycles$  such that
\begin{equation}\label{e:ExtraCycle}
 \w(\cyc)\ge \sum_{i \in\cyc } \sum_{\subcrit_\xi \ni i }\sr(A_\xi).
\end{equation}
(The right part runs over all nodes $i \in \cyc$ and for each $i$ accumulates the spectral radius $\sr(A_\xi)$ for the unique $\subcrit_\xi$ containing $i$.)

Let~$l$ and $m$ be respectively the smallest and  largest~$\xi$ such that~$\cyc\cap\subcrit_\xi\neq\emptyset$.
Note that~$\cyc$ is a simple cycle belonging~$\digr(A_l)$, since all its node occur in $\subcrit_\xi$ for some~$\xi\ge l$.
Assume first that $l =m $. Then, $\w(\cyc)\ge\len{\cyc}\sr(A_l)$ by~\eqref{e:ExtraCycle},  implying that  $\cyc$ is a simple cycle of~$\subcrit_l$.
%Since the simple cycles in~$\SCycles$ do not belong to the same~$s.c.c.$ of~$\subcrit_l$,
Since each simple cycle in~$\SCycles\cap\Cycles_l$ belongs to a different $s.c.c.$ of~$\subcrit_l$,
all nodes of~$\cyc$ appear in the same simple cycle~$\tilde{\cyc}$ in~$\SCycles \cap \Cycles_l$.
Since~$\cyc\notin\SCycles$, there is an arc of~$\cyc$ that does not belong to~$\tilde{\cyc}$.
Starting with this arc and going back  along the arcs of~$\tilde{\cyc}$, we build a
cycle of~$\subcrit_l$ shorter than~$\tilde{\cyc}$.
This contradicts the minimality of the length of~$\tilde{\cyc}$.

We are left with the case where~$l < m$.
Let~$\cyc_\xi$, with $\xi = l,m$,  be a simple cycle in~$\crit(A_\xi)$ that belongs to~$\SCycles$
such that  $\cyc\cap\cyc_\xi \neq\emptyset$, and let~$k_\xi$ be a node of this  nonempty intersection.
%
% Let~$\cyc_l$ (resp.~$\cyc_m$) be a cycle in~$\subcrit_l$ (resp. in~$\crit(A_m)$) that belongs to~ $\SCycles$,
%such that  $\cyc\cap\cyc_l\neq\emptyset$ (resp.~$\cyc\cap\cyc_m\neq\emptyset$) and let ~$k_l$ (resp.~$k_m$) be a node of the nonempty intersection.
It remains to show that
\begin{equation}\label{e:CSRmCSRl}
C_{\cyc_m}(S_{\cyc_m})^t R_{\cyc_m}\le C_{\cyc_l}(S_{\cyc_l})^t R_{\cyc_l}.
\end{equation}
Fix indices~$i$ and~$j$, for which $\left(C_{\cyc_m}(S_{\cyc_m})^t R_{\cyc_m}\right)_{i,j}\neq\zero$.
By Lemma~\ref{l:NonNormCSR}, there is a walk~$\gm_1\in\walksnode{i}{j}{\cyc_m}$ on~$\digr(A_m)$ such that
$\w(\gm_1)=\left(C_{\cyc_m}(S_{\cyc_m})^tR_{\cyc_m}\right)_{i,j}$.
In particular, $\gm_1$ intersects $\cyc_m$, and $\cyc_m$ intersects $\cyc$ at some~$k_m$.
 Insert $\cyc_m$ into
$\gm_1$ to get a walk~$\gm_2\in \walkslennode{i}{j}{t+\len{\cyc_m}}{k_m}$. Note that $\gm_2$ lives on~$\digr(A_m)$,
so it visits at most~$n-1$ different nodes, since all arcs incident to~$\subcrit_l$ do not belong to~$\digr(A_m)$.

By Corollary~\ref{c:TcRLin2}, applied to~$p=\len{\cyc}$ and~$i=k_m$, there is another walk~$\gm_3\in\walkslennode{i}{j}{t,\len{\cyc_m}}{k_m}$ on~$\digr(A_m)$,
of length at most~$(n-1)^2+\len{\cyc_m}-1$ (cf. Remark~\ref{rem:TcrUse}). But  $\len{\gm_3} = t\mod{\len{\cyc_m}}$, and thus $\len{\gm_3} \leq (n-1)^2<t$.
Inserting copies of~$\cyc_m$ into~$\gm_3$ at $k_m$ , we get a walk $\gm_4\in\walkslennode{i}{j}{t}{k_m}$. Namely,
$\gm_4$ is obtained from~$\gm_1$ by adding a copy of $\cyc$ and copies of~$\cyc_m$,
and removing cycles having  average weight at  most~$\sr(A_m)$,
which is the average weight of~$\cyc_m$. Therefore~$\w(\gm_4)\ge\w(\gm_1)=\left(C_{\cyc_m}(S_{\cyc_m})^tR_{\cyc_m}\right)_{i,j}$.

Now we  reduce~$\gm_4$ (cf. Remark~\ref{rem:TcrUse}), and then insert copies of~$\cyc$ at~$k_m$ to produce a new walk~$\gm_6\in \walkslennode{i}{~j}{t}{\cyc_l}$
for which
\begin{equation}\label{e:walkscomp}
 \w(\gm_6)\ge\w(\gm_4)\ge \left(C_{\cyc_m}(S_{\cyc_m})^t R_{\cyc_m}\right)_{i,j}.
\end{equation}
Note that $\gm_4$ is a walk on~$\digr(A_m)$, so it visits $\tilde{n}\le n-1$ different nodes,
at most~$n-\len{\cyc}$ of which do not belong to~$\cyc$, as $\gm_4$ is also a walk on~$\digr(A)$.
%We sperate to cases.
\begin{itemize}[leftmargin=.2in] \dispace
  \item[--]
 When $\tilde{n}<n-1$, we apply Proposition~\ref{p:TcRLin} with~$p=\len{\cyc}$ to get a walk~$\gm_5\in \walkslennode{i}{j}{t,\len{\cyc}}{\cyc}$
of length at most~$(\tilde{n}-1)\len{\cyc}+n-1\le (n-1)^2<t$. Then, we insert at least one copy of~$\cyc$ to get a walk~$\gm_6$ of length~$t$.
Since the average weight of~$\cyc$ is larger than~$\sr(A_m)$ by~\eqref{e:ExtraCycle}, and thus larger than the average weight of each cycle of~$\gm_4$, inequality
\eqref{e:walkscomp}~holds.

\item[--] The equality $\tilde{n}=n-1$ implies that $\cyc_l$ is a loop at~$k_l$. In this case, we apply Proposition~\ref{p:TcRLin} with~$p=1$ and get
$\gm_5\in \walksnode{i}{j}{k_m}$ of length at most~$n-1+n-1=2n-2$. Then we insert~$\cyc$ once and enough copies of the loop~$\cyc_l$ to get a walk~$\gm_6$ of length~$t$.
Since $\cyc_l$ is a critical cycle for~$A_l$, its average weight is~$ \sr(A_l)$, while $\sr(A_l)\ge \sr(A_m)$; thus \eqref{e:walkscomp}~holds.
\end{itemize}
Finally~\eqref{e:CSRmCSRl} follows from~\eqref{e:walkscomp} by Lemma~\ref{l:NonNormCSR}~$(i)$,
and we have proved that the sum~$\sum_{\cyc\in\SCycles} \len{\cyc}$ is not minimal, as long as
this sum is strictly larger than~$\trrk(A)$.
Thus, the inequality~$\frk(A^t)\le\trrk(A)$ follows from~\eqref{e:frk&Cycles},
applied to a minimal subcollection~$\SCycles$ that satisfies~\eqref{e:CSRExp}.
\end{proof}

\section{Semigroup identities of tropical matrices}\label{sec:4}
The following auxiliary  results  lead to Theorems~\ref{thm:Induction} and ~\ref{thm:IdExistence}.
We begin with an idea of Y.~Shitov~\cite{Shitov}, implemented in the following lemma.
\begin{lemma}\label{l:Shitov}
Let $A,B,C\in\Mn$ such that~$A=PQ$, where $P \in \Mat_{n\times k}$, $Q\in \Mat_{k \times n}$, $k < n$, and let  $w \in\{a,b\}^+$. Then
$(wa)\eval{AB}{AC}=P\big(w\eval{QBP}{QCP}\big)Q$.
\end{lemma}
\begin{proof}
Straightforward by induction on the length of the word ~$w$.
\end{proof}

To deal with matrices that cannot be factorized as above, we use  Theorem~\ref{thm:FullRank} which extends a result from~\cite{mxID}.
To this  ends additional results are needed, based on the following  conditions:
A pair of matrices $A,B\in \Mn$ and a word   $w \in \{a,b\}^+$ satisfy~\eqref{PR} if:
\begin{equation}\label{PR}
  \begin{array}{ll}
    \per(A)=\mtr(A), \   \ \per(B)=\mtr(B), \quad \text{and }%  \\[2mm]
    \trrk(w\eval{A}{B})= n.
  \end{array} \tag{PR}
\end{equation}

\begin{lemma}\label{l:ukkAB}
 Assume that \eqref{PR} holds for  $A,B\in \Mn$,  $w \in \{a,b\}^+$, and write $w = w_1 w_2 \cdots w_{\wlen(w)} $ as a sequence of letters.
  For each index~$i= 1,\dots,n$ we have
 \begin{equation}\label{eq:ukkAB}
  \big(w\eval{A}{B}\big)_{i,i} = \sum_{t=1}^{\wlen(w)} \big(w_t\eval{A}{B}\big)_{i,i}=\lno{a}{w} A_{i,i}+\lno{b}{w} B_{i,i},
  \end{equation}
i.e., the $i$'th diagonal entry of $w\eval{A}{B}$ is $\lno{a}{w} A_{i,i}+\lno{b}{w} B_{i,i}$.

 \end{lemma}

\begin{proof} Applying Theorem \ref{pr:perAB}, as \eqref{PR} holds, we have
$$
\begin{array}{ll}
\per\big(w\eval{A}{B} \big) =  \displaystyle{\sum_{t=1}^{\wlen(w)} \per\big(w_t\eval{A}{B}\big) }=
   \displaystyle{\sum_{t=1}^{\wlen(w)}  \mtr\big(w_t\eval{A}{B}\big)} \leq  \mtr\big(w\eval{A}{B} \big) ,
\end{array}
$$
implying the equality
\begin{equation}\label{eq:trwAB}
          \mtr\big(w\eval{A}{B}\big) =  {\sum_{t=1}^{\wlen(w)}  \mtr(w_t\eval{A}{B})},
         \end{equation}
 since   $\per\big(w\eval{A}{B} \big) \geq \mtr\big(w\eval{A}{B} \big)$.
Proposition~\ref{pr:WalkInterpret} obviously implies
\begin{equation}\label{eq:inukkAB}
 \big(w\eval{A}{B}\big)_{i,i} \ge \sum_{t=1}^{\wlen(w)} \big(w_t\eval{A}{B}\big)_{i,i}=\lno{a}{w} A_{i,i}+\lno{b}{w} B_{i,i}\; ,
\end{equation}
since the right hand side corresponds to the weight of the walk from~$\vx{i}$ to itself,  composed of loops only.
On the other hand, since~$w\eval{A}{B}$ is nonsingular, we have
$$\begin{array}{ll}
\displaystyle{\sum_{i=1}^n \big(w\eval{A}{B}\big)_{i,i}}
&=\mtr(w\eval{A}{B})
\ \overset{\eqref{eq:trwAB}}{=}  \ \displaystyle{\sum_{t=1}^{\wlen(w)}  \mtr(w_t\eval{A}{B})}\\
&=\lno{a}{w}  \mtr(A) +\lno{b}{w}  \mtr(B)
=\displaystyle{\sum_{i=1}^n \big(\lno{a}{w}  A_{i,i}+\lno{b}{w}  B_{i,i} \big)},
\end{array}$$
so that the inequality in~\eqref{eq:inukkAB} cannot be strict for any~$i$.  Hence, \eqref{eq:ukkAB} holds.
\end{proof}

\begin{lemma}\label{l:uABcyclFree}
Assume that  \eqref{PR} holds  for  $A,B\in \Mn$, $w \in \{a,b\}^+$.
For each entry $W_{i,j}\neq \zero$ of \mbox{$W = w\eval{A}{B}$} there is a 1-cyclic walk~$\gm_{i,j}$ on $\Auto(A,B)$, labeled by~$w$ of  weight $W_{i,j}$.

%Any entry of~$u\eval{A}{B}$ is the weight of a walk~$\gm$ which never goes back to a node which it has already left.
\end{lemma}

\begin{proof}
An $(i,j)$-entry of~$w\eval{A}{B}$ corresponds to the weight of a walk~$\gm_0:= \gm_{i,j}$ from $\vx{i}$ to $\vx{j}$ on $\Auto(A,B)$ labeled by $w$,
by Proposition~\ref{pr:WalkInterpret}.  Assume $\gm_0$ is not 1-cyclic, which means that $\gm_0$ returns to a node~$\vx{h}$ which it has already left. Let $\gm_1$ be the subwalk~$\gm_1$ of~$\gm_0$ which starts  at the first occurrence $\vx{h}$ and ends at  the last occurrence of~$\vx{h}$.
Let~$v$ be the factor of~$w$ labeling~$\gm_1$. Since~$v\eval{A}{B}$ is a factor of~$w\eval{A}{B}$, it follows from Proposition~\ref{pr:perAB} that
$\trrk(v\eval{A}{B})=n$, and~$A,B,v$ satisfy~\eqref{PR}. Hence, by Lemma~\ref{l:ukkAB} the walk $\gm_2$ labeled by~$v$ that stays at $\vx{h}$ has weight at least as that of $\gm_1$.
Then $\gm_1$  can be replaced by $\gm_2$ in~$\gm_0$ to obtain a walk $\gm_3$ that does not return to $\vx{h}$ after leaving $\vx{h}$ and whose weight is at least as that of~$\gm_0$.
Repeating  this process sequentially  for each recurrent node, we receive  a 1-cyclic walk~$\gm$ with weight at least as that of~$\gm_0$.
Since $\gm$ cannot have a strictly larger weight, we are done.
\end{proof}

We are now ready to prove:
\begin{theorem}\label{thm:FullRank}  Suppose that $\sid{u}{v}\in\Id(\Un)$, with $u,v\in \{a,b\}^+$, and that
 $A,B\in \Mn$ satisfy
 \begin{equation}\label{eq:PerTr}
    \per(A)=\mtr(A)\ \text{ and } \ \per(B)=\mtr(B);
\end{equation}
\begin{equation}\label{eq:rkU=rkV}
    \trrk(u\eval{A}{B})=\trrk(v\eval{A}{B})=n.
\end{equation}
Then, $u\eval{A}{B}=v\eval{A}{B}$. % \Com{Unclear}
\end{theorem}

\begin{proof}%[Proof of Theorem~\ref{thm:FullRank}]
We  prove that the following inequality holds for any entry $(i,j)$:
\begin{equation}\label{eq:uLev}
\left(u\eval{A}{B}\right)_{i,j}\le \left(v\eval{A}{B}\right)_{i,j}.
\end{equation}
The case of $\left(u\eval{A}{B}\right)_{i,j}=\zero$ is obvious.
Otherwise, Lemma~\ref{l:uABcyclFree} gives a 1-cyclic walk~$\gm_{i,j}$, i.e., $\gm_{i,j}$  never returns to a node which it has already left. Thus, the nodes of $\Auto(A,B)$ can be permuted, say by ~$\pi \in S_n$,   in a way that~$\gm_{i,j}$  has only arcs that go forward.
Let  $P := P_\pi$ be the  matrix associated to $\pi$, and let $T_A$ and~$T_B$ be the upper triangular matrices  obtained respectively from $P^{-1}AP$ and~$P^{-1}BP$  by setting all entries below the diagonal to~$\zero$. Then the matrix $P^{-1}AP$ satisfies
%\Com{What about zero entries?}
$$
\begin{array}{cc}
\big(u\eval{P^{-1}AP}{P^{-1}BP}\big)_{\pi(i),\pi(j)}
= \big(u\eval{T_A}{T_B}\big)_{\pi(i),\pi(j)},\end{array}
$$
and we can compute
$$
\begin{array}{ll}
\left(u\eval{A}{B}\right)_{i,j}
&=\left(u\eval{P^{-1}AP}{P^{-1}BP}\right)_{\pi(i),\pi(j)} \\[2mm] &
= \big(u\eval{T_A}{T_B}\big)_{\pi(i),\pi(j)}
= \big(v\eval{T_A}{T_B}\big)_{\pi(i),\pi(j)}\\[2mm]
&\le\left(v\eval{P^{-1}AP}{P^{-1}BP}\right)_{\pi(i),\pi(j)}
= \left(v\eval{A}{B}\right)_{i,j}.
\end{array}
$$
Thus~\eqref{eq:uLev} holds for each entry ~$(i,j)$. The reverse inequality holds by symmetry, so that~$u\eval{A}{B}=v\eval{A}{B}$.
\end{proof}

To apply Theorem \ref{l:uABcyclFree}, matrices which satisfy \eqref{eq:PerTr} should be detected; this is done by Lemma~\ref{lem:permAn}.

\begin{remark}\label{rem:determinantalRank}
Lemma~\ref{lem:permAn}, and consequently Theorem~\ref{thm:FullRank},  also hold if the maximality condition of tropical rank is replaced by maximality of the so-called determinantal rank, which is larger.
As well, modifying the notion of nonsingularity accordingly, Theorem~\ref{pr:perAB}  holds, cf. \cite[Theorem~2]{Shitov}. Nevertheless,  Theorem~\ref{pr:TropToFactor} holds for tropical rank, which suffices our needs.
\end{remark}

We can finally prove our main result:
\begin{theorem}\label{thm:Induction}
Given $n\in\N$, let  $\lcmn=\lcm(1,\dots,n)$. For any $t\ge \Wi$ and every $\sid{u}{v}\in\Id(\Mat_{n-1})$, where $u,v, p,\hp, q,\hq, r, \hr \in \{ a,b\}^+$, the following hold:
\begin{enumerate}\eroman
  \item If~$\sid{q}{r}\in\Id(\UT_{n})$, then~
\begin{equation}\label{eq:uv1}
 \sid{ua}{va}\substit{\big((qr)^{t}\big)\substit{a^{\lcmn}}{b^{\lcmn}}}{\big((qr)^{t}r\big)\substit{a^{\lcmn}}{b^{\lcmn}}} \in\Id(\Mat_{n});
\end{equation}

  \item If~$(p\hq p,p\hr p)\in\Id(\UT_{n})$, then
\begin{equation}\label{eq:uv2}
\sid{ua}{va}\substit{(w\hq p\big) \substit{a^{\lcmn}}{b^{\lcmn}}}{(w\hr p\big)\substit{a^{\lcmn}}{b^{\lcmn}}} \in\Id(\Mat_{n})
\end{equation}
with $w =(p\hq p \hr p)^{t}$.

\end{enumerate}

\end{theorem}

\begin{proof}
%Let us assume for a while that Theorem~\ref{thm:FullRank} and~\ref{pr:TropToFactor} are true and deduce Theorem~\ref{thm:Induction}.
$(i)$:
Let~$A,B\in \Mat_{n}$, %  and~$u,v,p,q,r,\hat{q},\hat{r}$ as in Theorem~\ref{thm:Induction}
and let
$$ X =  \big((qr)^{t}\big)\eval{A^\lcmn}{B^\lcmn}, \quad Y = \big((qr)^{t}r\big)\eval{A^\lcmn}{B^\lcmn}=XR, \quad \text{ with } R=r\eval{A^\lcmn}{B^\lcmn},$$
be matrices in $\Mat_{n} $.

\begin{itemize}[leftmargin=.2in]
  \item[--]

 If~$\frk\left(X\right)<n$, then $X=PQ$ for some matrices~$P\in\Mat_{n,n-1}$ and~$Q\in\Mat_{n-1,n}$. (Add columns and rows of~$\zero$, if $\frk\left(X\right)<n-1$.)
Hence $QP,QRP\in\Mat_{n-1}$,  and  using Lemma~\ref{l:Shitov} we obtain
\begin{equation*}
%(ua) \eval{\big((qr)^{t}\big)\eval{A^\lcmn}{B^\lcmn}}{\big((qr)^{t}r\big)\eval{A^\lcmn}{B^\lcmn}}
%= (va) \eval{\big((qr)^{t}\big)\eval{A^\lcmn}{B^\lcmn}}{\big((qr)^{t}r\big)\eval{A^\lcmn}{B^\lcmn}}
(ua) \eval{X}{Y}=P\big( u\eval{QP}{QRP}\big)Q=P\big(v\eval{QP}{QRP}\big)Q= (va) \eval{X}{Y},
\end{equation*}
since $\sid{u}{v}\in\Id(\Mat_{n-1})$ by assumption. Therefore,
\begin{equation}\label{eq:uava1}
%(ua) \eval{\big((qr)^{t}\big)\eval{A^\lcmn}{B^\lcmn}}{\big((qr)^{t}r\big)\eval{A^\lcmn}{B^\lcmn}}
%= (va) \eval{\big((qr)^{t}\big)\eval{A^\lcmn}{B^\lcmn}}{\big((qr)^{t}r\big)\eval{A^\lcmn}{B^\lcmn}}
(ua) \eval{X}{Y}= %P\big( u\eval{QP}{QRP}\big)Q=P\big(v\eval{QP}{QRP}\big)Q=
(va) \eval{X}{Y}.
\end{equation}

\item[--]
If $\frk\left(X\right)= n$, then  $\trrk\left((qr)\eval{A^\lcmn}{B^\lcmn}\right)=n$ by Theorem~\ref{thm:frk}, implying that
$$\trrk\left(q\eval{A^\lcmn}{B^\lcmn}\right)=\trrk\left(r\eval{A^\lcmn}{B^\lcmn}\right)=n$$
by Proposition~\ref{pr:perAB}.  Then $q\eval{A^\lcmn}{B^\lcmn}=r\eval{A^\lcmn}{B^\lcmn}$ by Theorem~\ref{thm:FullRank}, since $\sid{q}{r}\in\Id(\UT_{n})$.
Thus $X$ and~$Y$ are both powers of~$q\eval{A^\lcmn}{B^\lcmn}$, and hence commute.
This  implies~\eqref{eq:uava1},
since $\lno{a}{u} =\lno{a}{v} $ and $\lno{b}{u} =\lno{b}{v} $.
\end{itemize}
Therefore, \eqref{eq:uava1}~holds for any~$A,B\in\Mat_{n}$, which means that~\eqref{eq:uv1} holds true.

\pSkip
$(ii)$: The proof of~\eqref{eq:uv2} follows along the same lines of $(i)$.
\end{proof}

Replacing Theorem~\ref{thm:frk} in the proof by Proposition~\ref{pr:TropToFactor}, a similar result
is obtained, but with longer identities, in which $t$ is exchanged by~$\lcmn t$ and~$t\ge \Wi$ by~$t\ge 3n-2$.
Consequentially, by this change, only a subset of identities is produced.

\begin{theorem}
\label{thm:IdExistence} The monoid $\Mn$ satisfies a nontrivial semigroup identity for every $n \in \N$.
The length of this identity grows with~$n$ as~$e^{Cn^2+\operatorname{o}(n^2)}$ for some~$C\le 1/2+\ln(2)$.
\end{theorem}

\begin{proof}
  %Proof by induction on $n$.
  The case of $n=1$ is trivial, while \cite[Theorem 3.9]{IzMr} proves the case of $n=2$.
  For tropical triangular matrices there exists an identity $\sid{q}{r}\in\Id(\UT_{n})$ by \cite[Theorem 4.10]{trID}, or \cite[Theorem 0.1]{Okninski}.
  The proof then easily follows from Theorem \ref{thm:Induction} by induction.
  To bound the length,  note that ~$\sid{q}{r}\in\Id(\UT_{n})$ given by~\cite[Theorem 4.10]{trID}
  has length~$\wlen(q)=\wlen(r)=2^{n+\operatorname{o}(n)}$, while $\lcmn=e^{n+\operatorname{o}(n)}$ --
  a  fact that follows from the Prime Number Theorem.
\end{proof}

\begin{remark} Decreasing the length of ~$\sid{q}{r}\in\Id(\UT_{n})$ would  lead to a better bound on~$C$.
 Yet, with this method whose formulas includes $\lcmn$, such bound cannot be lower than  $1/2$.
\end{remark}

%% x^3y^3x^2yxy = 11
We immediately conclude the following:

\begin{corollary}\label{cor:rep.id}
Any semigroup which is faithfully represented  by $\Mn$  satisfies a nontrivial identity.\end{corollary}

\begin{example} Set $p=a^2b^2a[ab,ba]$. Then $\sid{pabp}{pbap}\in \Id(\UT_3)$  by \cite{trID.Er},
while $\Mat_2$ satisfies  an identity~$\sid{u}{v}$ of length 17 by~\cite{l17}.
Thus, by Theorem \ref{thm:Induction}.(ii),  $\Mat_3$ satisfies an identity of length~19,656, while Theorem \ref{thm:Induction}.(i)  gives a length~24,816.
In~\cite{Shitov} Shitov pointed out  that a matrix  $A\in\Mat_3$  has either determinantal rank~$3$ or factor rank at most~$2$. Consequently, for $\Mat_3$,  Remark~\ref{rem:determinantalRank} allows to omit exponent~$t$ in Theorem~\ref{thm:Induction}, which reduces the identity length to ~4,968 and~5,808, respectively.
\cite{Shitov} provides identities of  length~1,795,308.
\end{example}

%Since any finitely generated semigroup of tropical matrices has polynomial growth \cite{Ales, Simon},

\end{document}